\documentclass[12pt, twoside] {amsart}

\usepackage{amsfonts}
\usepackage{amssymb}
\usepackage{latexsym}
\usepackage{epsf,graphicx}
\usepackage{epsf,graphicx,latexsym,%
}

\newtheorem{example}{Example}[section]

\newcommand{\be} {\begin{eqnarray}}
\newcommand{\ee} {\end{eqnarray}}
\newcommand{\bep} {\begin{eqnarray*}}
\newcommand{\eep} {\end{eqnarray*}}

\newcommand {\Hol}{\mathop{\rm Hol}\nolimits}

\renewcommand {\Re}{\mathop{\rm Re}\nolimits}

\newcommand {\A}{\mathcal{A}}
\newcommand {\ad}{\mathop{\rm ad}\nolimits}

\newcommand {\Ff}{\mathcal{F}}

\newcommand{\R}{{\mathbb R}}
\newcommand{\N}{{\mathbb N}}
\newcommand{\B}{{\mathbb B}}
\newcommand{\C}{{\mathbb C}}

\newcommand {\D}{\mathbb{D}}

\newtheorem{remar}{Remark}[section]
\newtheorem{examp}{Example}[section]
\newtheorem{defin}{Definition}[section]
\newtheorem{corol}{Corollary}[section]
\newtheorem{propo}{Proposition}[section]
\newtheorem{theorem}{Theorem}[section]
\newtheorem{lemma}{Lemma}[section]

\newcommand{\rema}{\begin{remar}\rm}
\newcommand{\erema}{$\blacktriangleright$\end{remar}}

\newcommand{\exa}{\begin{examp}\rm}
\newcommand{\eexa}{$\blacktriangleright$\end{examp}}

\begin{document}

\title{Non-commutative holomorphic semicocycles}

\author[M. Elin]{Mark Elin}

\address{Department of Mathematics,
         Ort Braude College,
         Karmiel 21982,
         Israel}

\email{mark$\_$elin@braude.ac.il}

\author[F. Jacobzon]{Fiana Jacobzon}

\address{Department of Mathematics,
         Ort Braude College,
         Karmiel 21982,
         Israel}

\email{fiana@braude.ac.il}

\author[G. Katriel]{Guy Katriel}

\address{Department of Mathematics,
         Ort Braude College,
         Karmiel 21982,
         Israel}

\email{katriel@braude.ac.il}



\maketitle


\begin{abstract}
This paper studies holomorphic semicocycles over semigroups in the unit disk, which take values in
an arbitrary unital Banach algebra. We prove that every such
semicocycle is a solution to a corresponding evolution problem. We
then investigate the linearization problem: which semicocycles are
cohomologous to constant semicocycles? In contrast with the case
of commutative semicocycles, in the non-commutative case
non-linearizable semicocycles are shown to exist. Simple
conditions for linearizability are derived and are shown to be sharp.


\end{abstract}

\section{Introduction}
\setcounter{equation}{0}

The notion of (semi-)cocycle over a group or semigroup of
transformations is important in many areas of the theory of
dynamical systems (see, {\it{e.g.}}, \cite{K-H}), in particular in connection with the study of
non-autonomous differential equations (see, {\it{e.g.}},
\cite{Latus}). These objects, in addition to their intrinsic
interest, have applications to the study of invariant subspaces of
Banach spaces of holomorphic mappings and semigroups of weighted
composition operators on these spaces; see, {\it{e.g.}},
\cite{Latus,Konig, J-12}.

Several works have studies holomorphic semicocycles over semigroups of holomorphic self-mappings of the
open unit disk $\D$ of the complex plane $\C$, with
values in $\C$, which are {\it commutative} semicocycles. In
particular, it was shown in \cite{J-05} (see also \cite{Konig})
that each $\C$-valued semicocycle is smooth; the question whether
a semicocycle is a coboundary was considered in \cite{Konig,
J-97}.

Our aim in this work is to study {\it non-commutative}
semicocycles over semigroups of holomorphic self-mappings of $\D$,
a primary example of such semicocycles being those with values in
$M_n(\C)$, the space of $n\times n$ matrices. We wish, on the one
hand, to extend some results known for ($\C$-valued) commutative
semicocycles to the non-commutative case --- as we will see,
while some such extensions are straightforward, others require
some new ideas. On the other hand, we wish to stress some
essential differences between the commutative and the
non-commutative cases. We will see that interesting new phenomena
arise in the non-commutative case, even when we consider semicocycles over semigroups of the
{\it simplest} kind, that consist of linear functions.

\medskip

To introduce our basic notions, we start with some standard
notations. Let $\mathcal{D}$ be a domain in the complex plane $\C$
and let $\Omega$ be a domain in a complex Banach space. By
$\Hol(\mathcal{D},\Omega)$ we denote the set of all holomorphic
mappings on $\mathcal{D}$ with values in $\Omega$ and
$\Hol(\mathcal{D}):=\Hol(\mathcal{D,D})$ is the set of all
holomorphic self-mappings of $\mathcal{D}$. The open unit disk in
$\C$ is denoted by $\D$ and $\D_r=r\D,$ the open disk of radius
$r$. We also denote by $\A$ a unital Banach algebra over $\C$
equipped with the norm $\|\cdot\|$ such that the unit element
$1_\A$ satisfies $\|1_\A\|=1$. The set of all invertible elements
of $\A$ is denoted by $\A_*$. An important example is the algebra
of bounded linear operators on a Banach space; in particular, the
case of a finite-dimensional space corresponds to $\A=M_n(\C)$.

\begin{defin}\label{def-sg-hol}
    A family $\mathcal{F}
    =\left\{F_t\right\}_{t\ge0}\subset\Hol(\mathcal{D})$ is called a
    one-parameter continuous semigroup (semigroup, for short) on
    $\mathcal{D}$ if it is continuous in $(t,z)\in[0,\infty)\times\mathcal{D}$ and the following properties hold

    (i) $F_{t+s}=F_t \circ F_s$ for all $t,s\geq 0$;

    (ii) $F_0(z) =z$ for all $z\in \mathcal{D}$.
\end{defin}

It was proven in \cite{B-P} that {\it every one-parameter
continuous semigroup $\mathcal{F}=\left\{F_t\right\}_{t\ge0}$ on
$\D$ is differentiable with respect to its parameter $t\ge0$.} In
this case the limit
\begin{equation}\label{gener}
f(z) = \lim_{t\to 0^{+}} \frac{F_t(z) -z}t,\quad z\in\D,
\end{equation}
exists. The mapping $f\in \Hol(\D,\C)$ is called the {\sl
(infinitesimal) generator} of the semigroup~$\mathcal{F}.$ In
turn, the generated semigroup $\mathcal{F}$ can be reproduced as
the solution of the Cauchy problem
\begin{equation}  \label{nS1}
\left\{
\begin{array}{l}
\displaystyle
\frac{\partial u(t,z)}{\partial t}=f(u(t,z)) \vspace{2mm} \\
u(0,z)=z,%
\end{array}%
\right.
\end{equation}%
where we set $u(t,z)=F_{t}(z)$.

Another important way to represent and to study semigroups is by
means of a linearization model of the form
$F_t(z)=h^{-1}\left(G_t\circ h(z) \right)$, where
$h\in\Hol(\mathcal{D},\C)$ is a biholomorphic mapping, and
$\{G_t\}_{t\ge0}$ is an affine semigroup (see \cite{E-S-book} for
a survey of this problem). In particular, it is known that if for
some point $z_0\in\overline{\D}$, the semigroup generator
$f\in\Hol(\D,\C)$ satisfies $f(z_0)=0$ and $\Re f'(z_0)<0$, then
there is a spirallike function $h\in\Hol(\mathcal{D},\C)$, called
the K{\oe}nigs function, such that
    \begin{equation}\label{konigs}
F_t(z)=h^{-1}\left(e^{tf'(z_0)}h(z)\right).
    \end{equation}

We now introduce the central object to  be studied in this paper.

\begin{defin}\label{semicocycle}
    Let $\mathcal{F}=\{F_t\}_{t\ge0}\subset\Hol(\D)$ be a
    semigroup on the open unit disk. Let $\Gamma_t:\R^+\to \Hol(\D,
    \A)$ be such that the mapping $(t,z)\mapsto \Gamma_t(z) \in\A$ is
    continuous. The family $\left\{\Gamma_t\right\}_{t\ge0}$ is called
    a semicocycle over $\mathcal{F}$ if it satisfies the following
    properties:

    \begin{itemize}
        \item [(a)] the chain rule:
        $\Gamma_t(F_s(z))\Gamma_s(z) = \Gamma_{t+s}(z)$ for all $t,s\ge0$
        and $z\in\D$;

        \item [(b)] $\Gamma_0(z)=1_{\A}$ for every $z \in \D$.
    \end{itemize}
\end{defin}

Since both semigroup elements and semicocycle elements depend on
the real parameter $t\ge0$ and the point $z\in\D$, we usually
denote the derivative with respect to the parameter $t$ by
$\frac{\partial} {\partial t}$ while the derivative with respect
to $z$ is denoted by ${F_t}'(z),\ {\Gamma_t}'(z)$ and so on.

For $\A=\C$, the most natural example of a semicocycle over $\Ff$
is $\Gamma_t(z)={F_t}'(z)$. Another simple example is a
{\it{constant}} (not dependent on $z$) semicocycles of the form
$\Gamma_t(z)=e^{t B_0} $, where $B_0\in\A$.

In this work we study {\it inter alia} two general methods to
construct semicocycles:
\begin{itemize}
    \item[(i)] If $B:\D\rightarrow \A$ is a holomorphic mapping, a semicocycle $\Gamma_t(z)$ can be generated
    by solving the evolution problem
\begin{equation}\label{dd1}
\frac{d}{dt}\Gamma_t(z)=B(F_t(z))\Gamma_t(z),\;\;\;
\Gamma_0(t)=1_{\A}
\end{equation}
    (see Theorem \ref{th-sol-cauchy}(a) below), which can be considered a semicocycle analog of the Cauchy problem~\eqref{nS1} for
    semigroups.  Thus the mapping $B$ serves as a `generator' of a semicocycle.

    \item[(ii)] If $M: \D \rightarrow \A_*$ is holomorphic and $B_0\in\A$, then
\begin{equation}\label{rr0}
\Gamma_t(z)=M(F_t(z))^{-1}e^{tB_0}M(z)
\end{equation}
    is a semicocycle (see Lemma \ref{propo-gener} below).
    In view of the analogy between~\eqref{rr0} and \eqref{nS1}, representing a given semicocycle $\{\Gamma_t\}_{t\ge0}$ by~\eqref{rr0} will be called
    `linearization' of  $\{\Gamma_t\}_{t\ge0}$.  Note that the choice $B_0=0$ leads to semicocycles of
     the form $\Gamma_t(z)=M(F_t(z))M(z)^{-1}$, which are sometimes referred
    to as coboundaries; see, for example~\cite{J-97}.
\end{itemize}

These two constructions are of central importance: they
immediately suggest questions concerning their degree of
generality:

\begin{itemize}
    \item[(Q1)] Is every semicocycle differentiable? In the case of affirmative answer,
    is it generated, that is, can it be reproduced solving the Cauchy problem \eqref{dd1} for some $B\in\Hol(\D,\A)$?

    \item[(Q2)] Is every semicocycle linearizable, that is, can be  represented in the form \eqref{rr0}?
\end{itemize}

The last question is intimately connected to the problem of
classifying semicocycles up to the equivalence relation of
{\it{cohomology}}: a semicocycle $\Gamma_t$ is said to be
cohomologous to a semicocycle $\tilde{\Gamma}_t$ if there exists a
holomorphic mapping $M:\D\rightarrow \A_*$ (called `transfer
mapping') such that $\Gamma_t(z)=M(F_t(z))^{-1}\tilde{\Gamma}_t(z)
M(z)$ (see, {\it{e.g.}}, \cite{K-H}). Hence (Q2) can be
reformulated: is every semicocycle cohomologous to a constant
semicocycle?

These problems are the central focus of this work and are studied
in Sections~\ref{sect-diff} and \ref{sect-linea}.

For question (Q1), we will give a positive answer in
Section~\ref{sect-diff}, proving that every semicocycle is
generated by an appropriate holomorphic mapping $B$ (see Theorem
\ref{th-sol-cauchy} below). The main step here is to prove that a
semicocycle is automatically differentiable with respect to its
parameter~$t$. For the case $\A=\C$, this was proved in
\cite{J-05}, but the proof given there does not seem easy to
generalize to the non-commutative case. We therefore develop an
alternative proof, using an approach which appears to us simpler,
even in the case $\A=\C$. In this section, we also apply Theorem
\ref{th-sol-cauchy} to derive growth estimates for semicocycles.

When investigating question (Q2), we discover real differences
between the commutative and non-commutative cases. For $\A$
commutative, we show that the answer is positive, and an explicit
expression for the mapping $M$ is constructed in terms of the
`generator' of the semicocycle. The proof of this result, which
slightly generalizes the results of \cite{J-05} for the case
$\A=\C$, relies very strongly on commutativity. In fact, as soon
as we move to the non-commutative case, a simple example with
$\A=M_2(\C)$ shows that there are semicocycles which are not
linearizable (see Example~\ref{ex-not-rep}).

To understand the obstructions to linearization, we develop an
approach based on power series expansions (with coefficients in
$\A$), around the interior fixed point $z_0$ of the semigroup
$\mathcal{F}$. This allows us to show that the key factor is the
spectral behavior of the element $B_0=B(z_0)$, where $B$ is the
`generator' of $\Gamma_t$. We provide a simple and sharp condition
on the spectrum of $B_0$ under which the semicocycle is
linearizable. This condition, in the case $\A=M_n(\C)$, holds for
all matrices $B_0$ except for a `thin' set. Thus, while the
answer to (Q2) is negative as stated, it is positive in the
generic sense. Our results thus provide a quite precise answers to
the questions posed.

\bigskip

\section{Generation of semicocycles}\label{sect-diff}
\setcounter{equation}{0}

In this section we prove that every holomorphic semicocycle is
differentiable with respect to $t$, and can be generated as the
unique solution of an evolution problem the form \eqref{dd1}.

We will need some auxilliary results on general evolution problems in a Banach algebra $\A$.
Given $a:\R^{+}\to \A$, consider the evolution problem
\begin{equation}\label{cauchy_1}
\left\{
\begin{array}{l}
\displaystyle \frac{d v(t)}{d t} =a(t)v(t) \vspace{2mm} \\
v(0)=1_{\A},
\end{array}%
\right.
\end{equation}
where $v\in \R^+\to \A$. The general theory of such problems is well-known (see, for example, \cite{Kr, Pa} and references
therein), and the following theorem collects some basic facts.

\begin{theorem}\label{th-sol-cauchy_1}
Let $a(t)\in C(\R^{+}, \A)$. The following assertions hold.
\begin{itemize}
\item[(i)] The evolution problem~\eqref{cauchy_1} is equivalent to
the integral equation
\begin{equation}\label{integral-form}
v(t) =  1_\A + \int_0^t a(s)v(s) ds.
\end{equation}

\item[(ii)] The evolution problem~\eqref{cauchy_1} (hence, the
integral equation~\eqref{integral-form}) has a unique solution
$u\in C((0,\infty),\A)$. Moreover, defining for every element
$A\in \A$ the quantity
\begin{equation}\label{mu}\mu(A)=\lim_{t\to 0^{+}}\frac{1}{t}[\|1_{\A}+tA\|-1],\end{equation}
we have
\begin{equation}\label{estim-u}
\|u(t)\|_U\le \exp \left(\int_0^t \mu(a(s))ds \right).
\end{equation}

\item[(iii)] The element $u(t)$ is invertible for every $t\ge0$.
\end{itemize}
\end{theorem}

In particular, assertion (i) as well the uniqueness part of
assertion (ii) can be found in different monographs (see, for
example, \cite{Pa, Kr}), while estimate \eqref{estim-u} was proven
in \cite{AG1964}.

Our specific interest is to the special case of \eqref{cauchy_1}
in which the function $a$ has a `generator' structure. More
precisely, we assume that $a(t)=B(F_t(z))$, where
$\mathcal{F}=\{F_t\}_{t\ge0} \subset\Hol(\D)$ is a one-parameter
continuous semigroup and $B\in\Hol(\D,\A)$. In this case it is
natural to replace an $\A$-valued function $v(t)$ above by a
mapping ${ v(t,z) : \R^+ \times \D\to\A}$.
By~Theorem~\ref{th-sol-cauchy_1}, the corresponding evolution
problem has the unique solution $u(t):= u(t,z)$ for every fixed
$z\in\D$. It turns out that the solutions of these problems are
precisely all the semicocycles over the semigroup $\mathcal{F}$.

\begin{theorem}\label{th-sol-cauchy}
Let $\mathcal{F}=\left\{F_t\right\}_{t\ge0}\subset\Hol(\D)$ be a
semigroup on the open unit disk $\D$.

(a) Assume that $B\in\Hol(\D,\A)$. Denote by $u(t,z),\ t\ge0,\
z\in\D,$ the unique solution to the evolution problem
\begin{equation}\label{cauchy}
\left\{
\begin{array}{l}
\displaystyle \frac{d u(t,z)}{d t} =B(F_t(z))u(t,z) \vspace{2mm} \\
u(0,z)=1_{\A}.
\end{array}%
\right.
\end{equation}
Then the family $\{\Gamma_t(\cdot):=u(t,\cdot)\}_{t\ge0}$ is a
semicocycle over $\Ff$.

(b) Let $\left\{\Gamma_t\right\}_{t\ge0}$ be a semicocycle over
$\mathcal{F}$. Then $\Gamma_t(z)$ is differentiable with respect
to $t$ for every $z \in \D$. Defining
\begin{equation}\label{B}
B(z)=\frac{d}{dt}\Gamma_t(z)\Big|_{t=0},
\end{equation}
we have ${B\in\Hol(\D,\A)}$ and $\Gamma_t(z)$ is the solution to
the evolution problem~\eqref{cauchy}.
\end{theorem}

\begin{proof}
(a) Assume $u(t,z)$ satisfies \eqref{cauchy}. We first note that $u(t,\cdot)\in\Hol(\D,\A)$ by the theorem
on the differentiability of solutions of differential equations
with respect to parameters.

Fix $z\in \D$, $s\geq 0$ and define
\begin{equation*}\label{ch-ru-u1}
h(t)=u(t+s,z)-u\left(t, F_s (z)\right)u(s,z).
\end{equation*}
We have $h(0)=u(s,z)-u\left(0, F_s (z)\right)u(s,z)=0.$ Using
\eqref{cauchy} and the semigroup property of $\Ff$, we compute
\begin{eqnarray*}
 h^\prime(t)&=&\frac{\partial}{\partial t} u(t+s,z)-
\frac{\partial}{\partial t} u\left(t, F_s(z)\right)u(s,z)  \\
&=& B(F_{t+s}(z))u(t+s,z)-B(F_t(F_s(z))u(t,F_s (z))u(s,z)  \\
&=& B(F_{t+s}(z))u(t+s,z)-B(F_{t+s}(z))u(t,F_s (z))u(s,z)  \\
&=& B(F_{t+s}(z))\left(u(t+s,z)-u(t,F_s (z))u(s,z)\right)  \\
&=&B(F_{t+s}(z))h(t).
\end{eqnarray*}

Thus, the function $h(t)$ satisfies the linear differential
equation
\[
h^\prime(t)=B(F_{t+s}(z))h(t).
\]

Since $h(0) = 0$, the uniqueness theorem for linear differential
equations implies that
\begin{equation*}\label{cauchy_3}
h(t)=0\qquad \mbox{ for all }\ t\geq 0.
\end{equation*}
This proves that $\Gamma_t(z)=u(t,z)$ satisfies the chain rule. In
addition, ${\Gamma_0(z)=u(0,z)=1_\A}$, hence the family
$\{\Gamma_t\}_{t\ge0}$  is a semicocycle.


(b) Assume that $\left\{\Gamma_t\right\}_{t\ge0}$ is a
semicocycle. We fix $0<r<1$ and first prove that
$t\mapsto\Gamma_t(z)$ is differentiable for any $z\in \D_{r}$.

We define
\begin{equation}
\label{V}V(t,z)=\int_0^{t} \Gamma_s(z)ds.
\end{equation}
Note that:

(i) By the continuity of the function $s\mapsto\Gamma_s(z)$, the
mapping $V$ is differentiable with respect to $t$, namely,
$ \frac{\partial }{\partial t}V(t,z)=\Gamma_t(z).$

(ii) ${\Gamma_t}'(z)$, the derivative of $\Gamma_{t}(z)$ with
respect to $z$, is jointly continuous with respect to $(t,z)$, so
that differentiation of the expression (\ref{V})  with respect to
$z$ under the integral sign is valid, hence
 $V(t,z)$ is holomorphic with respect to $z$.

(iii) We have that $\frac{1}{t}V(t,z)\to 1_{\A}$ as $t\to 0$,
uniformly with respect to $z$ in compact sets of $\D$.
Therefore we can choose $t_0>0$ so that $\frac{1}{t_0}V(t_0,z)$,
hence also $V(t_0,z)$, is invertible for all $z\in
\overline{\D}_{r}$.

Using the semicocycle property
$\Gamma_{t+s}(z)=\Gamma_s(F_t(z))\Gamma_t(z)$, we get the identity
\begin{eqnarray}\label{id}
&&V(t_0,F_t(z))\Gamma_t(z)=\int_0^{t_0} \Gamma_s(F_t(z))
\Gamma_t(z)ds  \nonumber \\
&=&\int_0^{t_0} \Gamma_{t+s}(z)ds =\int_t^{t+t_0} \Gamma_{s}(z)ds
=V(t+t_0,z)-V(t,z).
\end{eqnarray}

Now, fixing $z\in\D_{r}$, we can choose $t_1>0$ so that $F_t(z)\in
\D_{r}$ whenever $0\leq t\leq t_1.$ Therefore by (iii) above,
$V(t_0,F_t(z))$ is invertible for $0\leq t \leq t_1$. Thus we can
rewrite \eqref{id} as
$$
\Gamma_t(z)=\left[V(t_0,F_t(z))\right]^{-1}\left[V(t+t_0,z)-V(t,z)\right].
$$
By (i) and (ii) above, both terms in the product on the right-hand
side are differentiable with respect to $t$. This implies that
$t\mapsto \Gamma_t(z)$ is differentiable on $[0,t_1]$ and, in
particular, at $t=0$. In fact, we can obtain an explicit equation
for the derivative $B(z)=\frac{d}{dt}\Gamma_t(z)\Big|_{t=0}$.
Differentiating both sides of (\ref{id}) and denoting by $f$ the
infinitesimal generator of $\Ff$, we obtain
\[
V_z(t_0,F_t(z))f(F_t(z))\Gamma_t(z)+
V(t_0,F_t(z))\frac{d}{dt}\Gamma_t(z)=\Gamma_{t+t_0}(z)-\Gamma_t(z),
\]
and then setting $t=0$ and rearranging we have
\[
B(z)=V(t_0,z)^{-1}[\Gamma_{t_0}(z)-1_{\A}- f(z)V_z(t_0,z)].
\]
In particular, this expression shows that the function $B$ is
holomorphic in $\D_{r}$. Since the choice of $0<r<1$ was
arbitrary, we have shown that $B$ is a well-defined holomorphic
function on $\D$. Moreover, we have, for any $t$,
\begin{eqnarray}\label{lim-2}
&&\frac{\partial}{\partial t} \Gamma_t(z) = \lim_{s\to0^+}\frac1s
\left(\Gamma_{t+s}(z)-
\Gamma_t(z)\right) \nonumber \\
&=& \lim_{s\to0^+}\frac1s \left(\Gamma_s(F_t(z))-
1_{\A}\right)\Gamma_t(z) =B(F_t(z))\Gamma_t(z).
\end{eqnarray}
Therefore, for every $z\in\D$, the $\A$-valued mapping $v(t)=
\Gamma_t(z)$ is differentiable with respect to $t$, and is the
unique solution of \eqref{cauchy}. Note, finally, that by setting
$t=0$ is \eqref{lim-2} we obtain the representation \eqref{B} of
the generator $B$.
\end{proof}

In the remainder of this section, we present a few results about
holomorphic semicocycles, which can be derived using the above
theorem.

Theorems \ref{th-sol-cauchy_1}(iii) and \ref{th-sol-cauchy}
immediately imply the following fact.
\begin{corol}\label{cor-inver}
Let $\left\{\Gamma_t\right\}_{t\ge0}\subset\Hol(\D,\A)$ be a
semicocycle over some semigroup $\mathcal{F}\subset\Hol(\D)$. Then
$\Gamma_t(z)$ is invertible  for any $t\ge0$ and $z\in\D$.
\end{corol}

The consequence of
Theorem~\ref{th-sol-cauchy}, which seems to be new even in the
one-dimensional case $\A=\C$, shows that the function $z\rightarrow \Gamma_t(z)$, for $t$ fixed,
satisfies a differential equation.
One can compare the formula below with a
known one for semigroup elements:
\[
{F_t}'(z)=\frac{f(F_t(z))}{f(z)}\,.
\]

\begin{corol}\label{cor-backw}
Let $\mathcal{F}=\left\{F_t\right\}_{t\ge0}\subset\Hol(\D)$ be a
semigroup generated by $f\in\Hol(\D,\C)$ and $\left\{\Gamma_t ,\
t\ge0\right\}$ be a semicocycle over $\mathcal{F}$. Denote
$B(z)=\left.\frac{\partial}{\partial t}\Gamma_t(z)\right|_{t=0}$.
Then
\[
{\Gamma_t}'(z)=\frac1{f(z)}\left[B(F_t(z))\Gamma_t(z)
-\Gamma_t(z)B(z)\right].
\]
\end{corol}

\begin{proof}
It follows by Theorem~\ref{th-sol-cauchy} that $\Gamma_t(z)$ is
differentiable at every point, hence the limit
\begin{eqnarray*}
\lim_{s\to0}\frac1s \left[ \Gamma_{t+s}(z)-\Gamma_t(z)\right] =
\lim_{s\to0}\frac1s \left[
\Gamma_{t}(F_s(z))\Gamma_s(z)-\Gamma_t(z)\right]
\end{eqnarray*}
exists. Since $\Gamma_t\in\Hol(\D,\A)$, we conclude that
\[
\Gamma_{t}(F_s(z))= \Gamma_t(z) +
{\Gamma_t}'(z)(F_s(z)-z)+o(F_s(z)-z).
\]
Therefore,
\begin{eqnarray*}
\lim_{s\to0}\frac1s \left[ \Gamma_{t+s}(z)-\Gamma_t(z)\right] =
\Gamma_t(z)B(z) +{\Gamma_t}'(z)f(z).
\end{eqnarray*}
On the other hand, the same limit equals $B(F_t(z))\Gamma_t(z)$ by
\eqref{lim-2}. So, the conclusion follows.
\end{proof}

We now use Theorem~\ref{th-sol-cauchy} to estimate the growth of
semicocycles. For this we need the following notation.
Given a semicocycle $\left\{\Gamma_t\right\}_{t\ge0}$ and a subset
$\mathcal{D}\subseteq\D$, denote
\[
\left\|\Gamma_t\right\|_{\mathcal{D}}:=\sup\left\{ \left\|
\Gamma_t(z) \right\|,\ z\in\mathcal{D}\right\} \in[0,\infty].
\]

\begin{theorem}\label{th-estim}
Let $\left\{\Gamma_t\right\}_{t\ge0}\subset \Hol(\D,\A)$ be a
semicocycle over a semigroup $\mathcal{F}=\{F_t\}_{t\ge0}$ and
$\mathcal{D}\subseteq\D$ be an $\mathcal{F}$-invariant domain. For
any $K \in \R$, the following assertions are equivalent
\begin{itemize}
  \item [(i)] $\|\Gamma_t(z)\|\leq e^{Kt}$ for all $z \in \mathcal{D}$;
\vspace{3mm}
  \item [(ii)]
$
\limsup\limits_{t\to 0^+} \frac{1}{t} \left[\left\|\Gamma_t
\right\|_\mathcal{D} -1\right]\leq K
$;
  \item [(iii)] $
\mu(B(z))\leq K
$ for all $z \in \mathcal{D}$, where ${B(z)=\frac{d}{dt}\Gamma_t(z)\Big|_{t=0}}$, and $\mu$ is defined by \eqref{mu}.
\end{itemize}
\end{theorem}

\begin{proof}
Suppose (i) holds. Then
\begin{eqnarray*}
\limsup\limits_{t\to 0^+} \frac{1}{t} \left[\left\|\Gamma_t
\right\|_\mathcal{D} -1\right]
\leq \limsup\limits_{t\to 0^+} \frac{1}{t} \left[e^{Kt}
 -1\right]=K,
\end{eqnarray*}
that is, (ii) is satisfied.

Assume now that (ii) holds. Fix
$z \in \mathcal{D}$. By definition of $B$, we have
\[
\Gamma_t(z)=1_\A+ tB(z)+o(t) \quad \mbox{ as } t \to 0.
\]
Therefore
\begin{eqnarray*}
&&\limsup\limits_{t\to 0^+} \frac{1}{t} \left[\left\|1_{\A}+tB(z)
\right\| -1\right]= \limsup\limits_{t\to 0^+} \frac{1}{t} \left[\left\|\Gamma_t(z)
\right\| -1\right]\\ &\leq& \limsup\limits_{t\to 0^+} \frac{1}{t} \left[\left\|\Gamma_t
\right\|_\mathcal{D} -1\right]\leq K.
\end{eqnarray*}
Thus (iii) follows.

Finally, if (iii) holds, then by Theorem \ref{th-sol-cauchy_1}(ii)
 we have (i).
\end{proof}

\begin{example}\label{exa-nonexp}
Denote
\[
\Gamma_t(z) =\frac {1+\sqrt{1-e^{-t} z}}   {1+\sqrt{1-z}}e^t,
\]
where the root branches are chosen such that $\sqrt{1}=1.$ A
direct calculation shows that the family $\{\Gamma_t\}_{t\ge0}$ is
a semicocycle over the linear semigroup
$\mathcal{F}=\{e^{-t}\cdot\} _{t\ge0}$. One can easily see that
\[
\left\|\Gamma_t\right\|_{\D}:=e^t\sup\limits_{z\in\D}\left|\frac
{1+\sqrt{1-e^{-t} z}}{1+\sqrt{1-z}} \right| \ \left\{
\begin{array}{l} \displaystyle \le \left(
1+ \sqrt{1+e^{-t}}\right) e^t \le 3e^t ,\vspace{2mm}\\
\ge |\Gamma_t(1)|=\left(1+\sqrt{1-e^{-t}}\right)e^t.
\end{array} \right.
\]
The upper inequality means that each $\Gamma_t$ is bounded on
$\D$. At the same time, the lower inequality implies that
\[
\frac{1}{t} \left[\left\|\Gamma_t\right\|_{\D} -1\right] \geq
\frac{\left(1+\sqrt{1-e^{-t}}\right)e^t-1}{t}=\frac{e^t-1}{t}+
\sqrt{\frac{e^t-1}{t}}\cdot \frac{e^{\frac{t}{2}}}{\sqrt{t}}\,.
\]
Thus
\[
\limsup_{t\to 0+} \frac{1}{t} \left[\left\|\Gamma_t\right\|_{\D}
-1\right]=\infty.
\]
Therefore by Theorem~\ref{th-estim}, $\{\Gamma_t\}_{t\ge0}$ does
not admit estimate of the form $e^{Kt}$.
\end{example}

It is interesting to compare Theorem~\ref{th-estim} with the
following assertion which was proven in \cite[Lemma 2.1]{Konig}
for the case where $\mathcal{D}=\D,$ the open unit disk, and
$\A=\C$. The proof in the general case is just a repetition of
that proof.

\begin{propo}\label{lem-konig}
Let $\left\{\Gamma_t\right\}_{t\ge0}$  be a semicocycle over some
semigroup and let $\mathcal{D}$ be an $\mathcal{F}$-invariant
domain. The following assertions are equivalent:
\begin{itemize}
  \item [(i)]
$\left\|\Gamma_t\right\|_{\mathcal{D}}<\infty$ for every
$t\ge0$.\vspace{2mm}
  \item [(ii)]  $\limsup\limits_{t\to0^+}
\left\|\Gamma_t\right\|_\mathcal{D}<\infty$;
  \item [(iii)] $\left\|\Gamma_t\right\|_\mathcal{D} \le Me^{Kt}$ for some real
$M$ and $K$.
\end{itemize}
\end{propo}
This criterion of boundedness is analogous to a known estimate for
strongly semigroups of linear operators while inequality (i) in
Theorem~\ref{th-estim} is similar to estimates of uniformly
continuous semigroups; see, for example, \cite{Pa, YK}.

\begin{example}\label{exa1-exp1}
Let $\A=\C$ and $B(z)=\displaystyle \frac1{1-z}\,.$ Clearly, the
growth of semicocycles depends on the asymptotic behavior of
semigroups.

We start with the linear semigroup $\mathcal{F}=\{e^{-t}\cdot \}
_{t\ge0}$ converging to the interior point $z_0=0$. In this case
the invariant sets are just disks $\D_r,\ r\le1$.

Solving the evolution problem~\eqref{cauchy}, one finds that
\[
\Gamma_t(z) =\frac {e^t -z}   {1-z}\,.
\]
Fix some $r<1.$ Then
\[
\left\|\Gamma_t\right\|_{\D_r}:=\sup\limits_{z\in\D_r}\left| \frac
{e^t -z}{1-z} \right| =\frac{e^t-r}{1-r}
\]
and
    \[
\limsup_{t\to 0+} \frac{1}{t} \left[\left\|\Gamma_t\right\|_{\D_r}
-1\right]=\frac1{1-r}\,.
    \]
Therefore $\|\Gamma_t\|_{\D_r}\le \exp\left(K_rt\right)$ with
$K_r=\frac{1}{1-r}$. At the same time, for every $t>0$, $\Gamma_t$
is unbounded on the whole unit disk $\D$.

Now we consider the affine semigroup on $\D$ defined by $$F_t(z) =
1- (1-z)e^{-t}.$$ It converges to the boundary point $z_0=1$ and
the invariant sets are disks tangent to the unit circle at this
point. Solving again~\eqref{cauchy}, we obtain
\[
\Gamma_t(z):=\exp \left[ \frac{e^t -1}  {1-z} \right].
\]
In this case, for every $z_0\in\D$,
\[
\Gamma_t(F_s(z_0))=\exp \left[ \frac{e^s(e^t -1)} {1-z_0} \right].
\]
So, any $\Gamma_t$ is unbounded on every semigroup trajectory, and
hence cannot satisfy estimate of the form $Me^{Kt}$ considered in
Proposition~\ref{lem-konig}.
\end{example}

\bigskip

\section{Linearization of semicocycles}\label{sect-linea}
\setcounter{equation}{0}

\subsection{Statement of the problem}

As above, let $\mathcal{F}=\{F_t\}_ {t\ge0}$ be a semigroup on the
open unit disk $\D$. In Section~\ref{sect-diff}, we have seen that
each semicocycle can be generated by solving a non-autonomous
linear differential equation. A different way to obtain
semicocycles is based on the following simple assertion.

\begin{lemma}\label{propo-gener}
Let $M\in\Hol(\D,\A_*)$ and $B_0\in \A$. Then the family
$\left\{\Gamma_t\right\}_{t\ge0}\subset\Hol(\D,\A_*)$ defined by
\begin{equation}\label{rep}
\Gamma_t(z)=M(F_t(z))^{-1}e^{tB_0}M(z)
\end{equation}
is a semicocycle over $\mathcal{F}$.
\end{lemma}

\begin{proof}
Obviously, $\Gamma_t\in\Hol(\D,\A)$, is jointly continuous with
respect to $(t,z)$ and $\Gamma_0(z)=1_{\A}$. In addition, we have
    \begin{eqnarray*}
    \Gamma_t(F_s(z))\Gamma_s(z)&=&M(F_t(F_s(z)))^{-1}e^{tB_0}
    M(F_s(z))M(F_s(z))^{-1}e^{sB_0}M(z)    \\
    &=&M(F_t(F_s(z)))^{-1}e^{(t+s)B_0}M(z)=\Gamma_{t+s}(z). \end{eqnarray*}
So, $\left\{\Gamma\right\}_{t\ge0}$ is a semicocycle.
\end{proof}

\medskip

A semicocycle obtained in this way is cohomologous to the constant
semicocycle $\tilde{\Gamma}_t(z)=e^{t B_0}$.

Our central question in this section is:\vspace{3mm}

$\bullet$ \ Which semicocycles admit the representation
\eqref{rep}, for an appropriate choice of $B_0\in\A$ and
$M\in\Hol(\D,\A_*)$?
\vspace{3mm}

We will call a representation of the form \eqref{rep} a
`linearization' of the semicocycle $\Gamma_t$, and if such a
representation exists we will say that the semicocycle is
`linearizable'. As explained in Introduction, this terminology is
intended to convey an analogy with the linearization of
semigroups.

It was shown in \cite{Konig, J-97} that if $\A=\C$ and the
semigroup $\Ff$ has no interior fixed point then each semicocycle
can be represented in the form $\Gamma_t(z)=M(F_t(z))^{-1}M(z)$,
where $M\in\Hol(\D,\C\setminus\{0\})$. For the case where $\Ff$
has an interior fixed point $z_0\in\D$ the same representation was
deduced with $M$ analytic in $\D\setminus\{z_0\}$ --- but not in
$\D$. The following example shows that the generalized
representation~\eqref{rep} is more relevant since it allows to
describe a larger family of semicocycles using functions analytic
in the whole disk.

\begin{example}
Let a semigroup $\mathcal{F}= \{F_t\}_{t\ge0}\subset\Hol(\D)$
preserve zero and be generated by a function $f\in\Hol(\D,\C)$
with $\Re f'(0)<0$. Define the family
$\left\{\Gamma_t\right\}_{t\ge0}$ as follows:
 \[
 \Gamma_t(z)=\left(\frac{F_t(z)}z\right)^\beta
 \left(m(F_t(z))\right)^{-1}m(z),
 \]
where $\beta$ is any real number and a function
$m\in\Hol(\D,\A_*)$ is normalized by $m(0)=1_{\A}$. It can be
verified directly that this family is a semicocycle
over~$\mathcal{F}$. In order to represent it by formula
\eqref{rep}, recall (see \eqref{konigs}) that for such semigroup
there is the unique solution $h$ of the Schr\"oder functional
equation
 \[
 h(F_t(z))=e^{tf'(0)}h(z)
 \]
normalized by $h'(0)=1$. Hence, taking $\displaystyle
M(z)=\left(\frac{h(z)}{z}\right)^\beta m(z)$, we get
 \[
 \Gamma_t(z)=e^{t\beta f'(0)} \left(M(F_t(z))\right)^{-1}M(z),
 \]
which coincides with \eqref{rep}.
\end{example}

We note that the representation \eqref{rep}, if it exists, is not unique: if we
choose any $A\in \A_*$ and define
      \begin{eqnarray}\label{A}
\tilde{M}(z)=AM(z),\quad \tilde{B}_0=AB_0A^{-1},
      \end{eqnarray}
then
    \begin{eqnarray*}
    \tilde{\Gamma}_t(z)=\tilde{M}(F_t(z))^{-1}e^{t\tilde{B}_0}\tilde{M}(z)&=&
    M(F_t(z))^{-1}A^{-1}e^{t AB_0 A^{-1}}AM(z) \nonumber  \\
    &=&M(F_t(z))^{-1}e^{t B_0}M(z)= \Gamma_t(z).
    \end{eqnarray*}
This fact allows us to chose the representation in a convenient
form depending on the location of the Denjoy--Wolff point of
$\Ff$.

$\bullet$  If the semigroup $\Ff$ has a boundary Denjoy--Wolff
point $z_0\in\partial\D$, then there exists a univalent function
$h$ that satisfies Abel's functional equation
\[
h\circ F_t =h+t\quad\mbox{on\ }\D
\]
(see, {\it e.g.}, \cite{E-S-book}). Hence, if a semicocycle
$\left\{\Gamma_t\right\}_{t\ge0}$ over $\Ff$ admits representation
\eqref{rep}, then
\begin{eqnarray*}
\Gamma_t(z)&=&M(F_t(z))^{-1}e^{tB_0}M(z) \\
& =& M(F_t(z))^{-1}e^{(h(F_t(z))-h(z))B_0}M(z) \\
&=& M_1(F_t(z))^{-1}M_1(z),\quad \mbox{where}\quad
M_1(z)=e^{-h(z)B_0}M(z).
\end{eqnarray*}
Recall that semicocycles which can be represented in the form
\eqref{rep} with $B_0=0$, that is
$\Gamma_t(z)=M(F_t(z))^{-1}M(z)$, are known as coboundaries. Thus
\begin{propo}\label{pro}
If $\Ff$ has a boundary Denjoy--Wolff point and a semicocycle over
$\Ff$ admits representation of the form \eqref{rep}, then this
semicocycle is a coboundary.
\end{propo}
This explains why in the above-mentioned result from \cite{Konig,
J-97} there was no need to introduce a term $e^{B_0 t}$ in the
representation in the case of a boundary Denjoy--Wolff point.

$\bullet$  On the other hand, if $z_0\in\D$ is an interior fixed
point of $\mathcal{F}$, then we may choose $A=M(z_0)^{-1}$ in
\eqref{A}, so that we obtain $\tilde{M}(z_0)=1_{\A}$. Therefore we
may assume, without loss of generality, that the mapping $M$ in
\eqref{rep} satisfies the normalization
\begin{equation}\label{norm}
M(z_0)=1_{\A},
\end{equation}
which we will henceforth do.

As above, denote $
B(z)=\frac{d}{dt}\Gamma_t(z)\Big|_{t=0}.$ From
\eqref{rep}--\eqref{norm} we have
$$\Gamma_t(z_0)=M(z_0)^{-1}e^{tB_0}M(z_0)= e^{tB_0},$$
hence
\[
    B(z_0)=\frac{d}{dt}\Gamma_t(z_0)\Big|_{t=0}=
    \frac{d}{dt}e^{tB_0}\Big|_{t=0}=B_0.
\]
Thus $B_0$ is uniquely determined by
\begin{equation}\label{B0}
B_0=B(z_0).
\end{equation}

From this and \eqref{A} we see that a necessary condition for a
semicocycle $\Gamma_t$ to be a coboundary is that its generator
$B$ satisfy $B(z_0)=0$.

We now show that it is not true, in general, that semicocycles are linearizable.

\begin{example}\label{ex-not-rep}
Let $\A$ be the Banach algebra $M_2(\C)$ of $2\times2$ matrices,
equipped with the operator norm. Consider the linear semigroup
$\mathcal{F}= \{e^{-t}\cdot\}_{t\ge0}$ acting on $\D$ and define the
family $\left\{\Gamma_t\right\}_{t\ge0}$ by
 \[
 \Gamma_t(z)=\left( \begin{array}{cc}
e^t &  zt e^t\\
0 & e^{2t}
\end{array} \right).
 \]
A straightforward computation shows that
$\left\{\Gamma_t\right\}_{t\ge0}$ is a semicocycle
over~$\mathcal{F}$. Differentiating this semicocycle at $t=0$, one
sees that its generator is
 $\displaystyle B(z)= \left(\begin{array}{cc}
1 & z \\
0 & 2
\end{array}  \right), $
so that
 $\displaystyle B_0=B(0)= \left(\begin{array}{cc}
1 & 0 \\
0 & 2
\end{array}  \right). $

Assume that $\Gamma_t(z)$ can be represented by \eqref{rep} with
some
\[
M(z)=\left( \begin{array}{cc}
\eta_{11}(z)  &  \eta_{12}(z)\\
\eta_{21}(z) & \eta_{22}(z)
\end{array} \right) \in\Hol(\D,\A_*),
\]
that is, $M(F_t(z))\Gamma_t(z)=e^{tB_0}M(z)$. Multiplying these
matrices and equating elements of the first row we get
\begin{eqnarray*}
\eta_{11}(e^{-t}z)\cdot e^t &=& e^{t} \cdot\eta_{11}(z)  \\
\eta_{11}(e^{-t}z) \cdot z te^t
+\eta_{12}(e^{-t}z)\cdot e^{2t} &=& e^{2t}\cdot \eta_{12}(z).
\end{eqnarray*}
By taking $t\rightarrow\infty$ in the first equality and using the
normaliztion $\eta_{11}(0)=1$, we conclude that $\eta_{11}(z)$ is
a constant function, $\eta_{11}=1$, so that the second equality
can be written as
$$ z \cdot t\cdot e^{-t}
+\eta_{12}(e^{-t}z) =  \eta_{12}(z),$$
which leads to a contradiction when taking $t\rightarrow\infty$, since $\eta_{12}(0)=0$.
Therefore, $\Gamma_t(z)$ cannot be represented by \eqref{rep}.
\end{example}

In the following we will analyze the problem of linearization, and
gain a good understanding of the obstructions to linearizability,
which in particular will give us a deeper explanation for the
non-linearizability in the above example.

We will first show that in the case of semicocycles whose range
$\A$ is {\it{commutative}}, every semicocycle is linearizable, and
the linearizing mapping $M$ can be represented in a simple
explicit way. Moving to the non-commutative case, we will obtain a
sufficient condition for linearizability, depending only on the
value $B_0=B(z_0)$ of the semicocycle generator at the interior
fixed point of the semigroup. This condition will also be shown to
be sharp. In the case where $\A$ is finite-dimensional, it will be
seen that the condition for linearizability is equivalent to a
simple condition on the eigenvalues of $B_0$, from which it will
follow that it holds for `generic' $B_0\in \A$.

\subsection{The commutative case}

In this subsection we assume that the algebra $\A$ is commutative.
We show that every semicocycle is linearizable. Our results
somewhat generalize results in \cite{Konig, J-97}, which
considered $\A=\C$.

We begin with the case where $\Ff$ has no interior fixed point.
\begin{theorem}\label{th-konig}
Let $\mathcal{F}=\left\{F_t\right\}_{t\ge0}\subset\Hol(\D)$ be a
semigroup generated by $f\in\Hol(\D,\C)$ with no  interior fixed
point. Let $\{\Gamma_t\}_{t\ge0} \subset\Hol(\D,\A)$ be a
semicocycle over $\mathcal{F}$, generated by $B$. Then
$\Gamma_t(z)$ is linearizable, namely, $\Gamma_t(z) =
M(F_t(z))^{-1}M(z),$  where
\begin{equation}\label{M}
M(z) = \exp\left( -\int_0^z\frac1{f(w)}B(w)dw  \right).
\end{equation}
\end{theorem}
\begin{proof}
We first note that $M(z)$ in \eqref{M} is well-defined and
$M\in\Hol(\D,\A_*)$. Denote
\begin{eqnarray*}
\widetilde{\Gamma}_t(z) := M(F_t(z))^{-1}M(z) = \exp\left(
\int_z^{F_t(z)}\frac1{f(w)}B(w)dw \right) .
\end{eqnarray*}
Because $\A$ is commutative, differentiation of this equality
gives
$$\frac{\partial \widetilde{\Gamma}_t(z)}{\partial t}=
B(F_t(z))\widetilde{\Gamma}_t(z).$$

Since $\Gamma_t$ satisfies the same differential equation and
$\widetilde{\Gamma}_0(z)=\Gamma_0(z)=1_\A$, the uniqueness theorem
completes the proof.
\end{proof}

We proceed with the case where $\Ff$ has an interior fixed point.
Here, in contrast with \cite{Konig, J-97}, we prove that every
semicocycle is linearizable by a function $M$ holomorphic in the
whole unit disk $\D$.

\begin{theorem}\label{th-present-Gamma}
Let $\mathcal{F}=\left\{F_t\right\}_{t\ge0}\subset\Hol(\D)$ be a
semigroup with an attractive fixed point $z_0\in \D$. Let $\A$ be commutative and
$\{\Gamma_t\}_{t\ge0} \subset\Hol(\D,\A)$ be a semicocycle over
$\mathcal{F}$. Then $\Gamma_t(z)$ is
linearizable, namely,
\[
\Gamma_t(z) =  M(F_t(z))^{-1}e^{tB_0}M(z),
\]
where $B(z)=\frac{d}{dt}\left.\Gamma_t(z)\right|_{t=0}$, $B_0=B(z_0)$, and
 $$M(z)=\exp\left[
\int_0^\infty \left(B(F_t(z))-B_0\right)dt\right].$$
\end{theorem}

\begin{proof}
Without loss of generality we can assume that $z_0=0.$ 
Denote by $f\in\Hol(\D,\C)$ the generator of $\Ff$, so that $f(0)=0$
and $\lambda:=-f'(0)$ satisfies $\Re\lambda>0$.
Therefore, we have
\[
   |F_t(z)| \le |z| \exp\left(-\lambda t\frac{1-|z|}{1+|z|}
      \right),\quad z\in\D,\quad t\ge0
\]
(see, for example, \cite[Proposition 4.4.2]{SD}).

We already know by Theorem~\ref{th-sol-cauchy} that
$\{\Gamma_t\}_{t\ge0}$ is differentiable, hence
$ B(z):=\left.\frac{d}{dt} \Gamma_t(z)
\right|_{t=0}$ exists.  First we show that the integral
\begin{equation}\label{N}
N(z):=\int_0^\infty \left(B(F_t(z))-B_0\right)dt
\end{equation}
converges for any $z\in\D$ and defines a holomorphic function
$N\in\Hol(\D,\A_*)$.

Indeed, fix $R\in(0,1)$. There exists $L>0$ such that
$\|B(z)-B_0)\|\le L$ whenever $|z|<R$. Then by the Schwarz Lemma
\[
\|B(z)-B_0\| \le \frac{L|z|}{R}\,,\quad z\in\D_{R}.
\]

For any $z\in\B$ there is $t_1>0$ such that $|F_{t_1}(z)|<R$. Then
for all $t>0$,
\begin{eqnarray*}
\|B(F_{t_1+t}(z))-B_0\| &\le& \frac{L |F_{t_1+t}(z)|}{R} \\
&\le& \frac{L |F_{t_1}(z)|}{R} \exp\left(
-\lambda t\frac{1- |F_{t_1}(z)|}{1+|F_{t_1}(z)|} \right) \\
&=& M \exp\left(-\lambda t\frac{1-R}{1+R} \right).
\end{eqnarray*}
Therefore, the integral \eqref{N} is absolutely convergent at any
point $z\in\D$ and uniformly convergent on $\D_{R}$.

We now consider the integral
\begin{eqnarray}\label{int-der}
\int_0^\infty{\left(B(F_t(z))-B_0\right)' dt} = \int_0^\infty{
B'(F_t(z)) {F_t}'(z) dt}.
\end{eqnarray}
To estimate it, we note that by the Cauchy inequality $\|B'(z)\|
\le \frac{L}{R-R_1}$ whenever $|z|< R_1,\ R_1\in(0,R)$. We choose
$R_1$ such that such that $\Re f'(z)\le\frac {-\lambda}{2}$ for all
$z\in\D_{R_1}$.

In addition, differentiating \eqref{nS1} we get that the function
$v$ defined by $v(t):={F_t}'(z)$, is the solution of the evolution
problem
\[
\left\{
\begin{array}{l}
\displaystyle
\frac{\partial v(t)}{\partial t}=f'(F_t(z)) v(t) \vspace{2mm} \\
v(0)=1.
\end{array}%
\right.
\]
Therefore $\left|{F_t}'(z)\right| \le e^{\frac{-\lambda t}2}$ (for
instance, it follows from assertion (ii) of
Theorem~\ref{th-sol-cauchy_1}), hence the integral \eqref{int-der}
converges uniformly on $\D_{R_1}$. Therefore one can differentiate
\eqref{N} as follows:
\begin{eqnarray*}
N'(z) &=& \int_0^\infty B'(F_t(z)) {F_t}'(z) dt \\
&=& \frac1{f(z)} \int_0^\infty B'(F_t(z)) f(F_t(z)) dt \\
&=& \frac1{f(z)} \int_0^\infty\left[\frac{\partial}{\partial t}
B(F_t(z))\right]dt = \frac{B_0-B(z)}{f(z)}.
\end{eqnarray*}
Thus
\begin{equation}\label{gamma1}
B(F_s(z))=B_0-N'(F_s(z))f(F_s(z)).
\end{equation}

Note that since $\A$ is commutative and using the uniqueness of
solutions of differential equations,
\[
\Gamma_t(z)=\exp\left(\int_0^t B(F_s(z))ds \right).
\]
(because the right-hand side in the last formula satisfies
\eqref{cauchy}). Now using \eqref{gamma1}, we rewrite this as
follows:
\begin{eqnarray*}
\Gamma_t(z)=\exp\left(\int_0^t
\left(B_0 -N'(F_s(z))f(F_s(z)) \right)ds \right) \\
= \exp\left[ tB_0- N(F_t(z))+N(z)\right].
\end{eqnarray*}
This completes the proof.
\end{proof}

\subsection{The general case}
Henceforth in this section we will assume that the semigroup has
an interior fixed point $z_0\in \D$, and thus we can assume the
normalization \eqref{norm}, hence \eqref{B0} holds. As we have
seen in Example \ref{ex-not-rep} above, general non-commutative
semicocycles are not always linearizable. We will now analyze the
problem of finding a linearizing mapping $M$, and develop a
procedure to construct it using a power series expansion, which
will also reveal the nature of the obstructions to
linearizability.

We write \eqref{rep} in the form
\begin{equation}\label{repp}
M(F_t(z))\Gamma_t(z)=e^{tB_0}M(z),
\end{equation}
and consider it as a linear functional equation, which is to be
solved for $M\in \Hol(\D,\A_*)$.

As a first step, it will be useful to simplify \eqref{repp} as
follows. We denote by $h\in\Hol(\D,\C)$ the K{\oe}nigs function
associated to the semigroup, so that
$$h(F_t(z))=e^{-\lambda t}h(z),$$
where $\lambda=-f'(0)$, with $\Re(\lambda)>0$, and $h(z_0)=0$,
$h'(z_0)=1$. Recall that $\Omega=h(\D)$ is a spirallike domain.
Set $z=h^{-1}(w),\ w\in\Omega,$ in \eqref{repp} to obtain
\[
M(h^{-1}(e^{-\lambda t} w ))\Gamma_t(h^{-1}(w))=
e^{tB_0}M(h^{-1}(w)),
\]
or, setting
 \begin{eqnarray}\label{rr1}
 &&\gamma_t(w)=\Gamma_t(h^{-1}(w)),\qquad m(w)=M(h^{-1}(w)),
 \end{eqnarray}
 \begin{equation}\label{rr2}
 m(e^{-\lambda t }w)\gamma_t(w)=e^{tB_0}m(w).
 \end{equation}
Thus, finding a representation of $\Gamma_t$ in the form
\eqref{rep} is equivalent to finding $m\in \Hol(\Omega,\A_*)$ such
that \eqref{rr2} holds. We shall therefore treat the problem of
solving \eqref{rr2} for $m\in \Hol(\Omega,\A_*)$.

Our strategy is to solve first \eqref{rr2} for $m$ {\it{locally}}
in a neighborhood of $w=0$, and then note that any such solution
automatically extends to all of~$\Omega$.

\begin{lemma}\label{local}
    Let $r_0$ be such that $\D_{r_0}\subset \Omega=h(\D)$.
    A holomorphic mapping $m:\D_{r_0}\rightarrow \A$, which satisfies \eqref{rr2} for all $w\in\D_{r_0}$
    and $t\geq 0$,
    can be extended to a holomorphic mapping satisfying \eqref{rr2} for all $w\in \Omega$.
\end{lemma}

\begin{proof}
    Fix $r>r_0$. Since the domain $\Omega$ is spirallike, the set  $\Omega\cap \D_r$ is connected.
    To extend $m$ to $\Omega\cap \D_r$, choose $t_0>0$ so that $|e^{-\lambda t
    _0}|r<r_0$,
    and define
    \[
    \tilde{m}(w)=e^{-t_0B_0}m(e^{-\lambda t _0}w)\gamma_{t_0}(w),\quad w\in \Omega\cap \D_r.
    \]
    In view of \eqref{rr2}, $\tilde{m}(w)=m(w)$ for $w\in\D_{r_0}$, so that $\tilde{m}$ extends $m$ holomorphically to
    $\Omega\cap\D_r$. By holomorphicity, the equation \eqref{rr2} will continue to hold in the larger domain.
    Since $r>0$ is arbitrary,
    we can extend $m$ to the whole of $\Omega$.
\end{proof}

We now begin the process of solving \eqref{rr2} locally around
$w=0$. Setting
$$b(w)=B(h^{-1}(w)), $$
we note that
\begin{eqnarray}\label{de1}
\frac{d}{dt}\gamma_t(w)&=&\frac{d}{dt}\Gamma_t(h^{-1}(w)) =
B\left(F_t(h^{-1}(w))\right) \Gamma_t\left(h^{-1}(w)\right) \nonumber\\
&=&  B\left(h^{-1}(e^{-\lambda t}w)\right) \gamma_t(w) = b(e^{-\lambda
t}w)\gamma_t(w).
\end{eqnarray}
Differentiating \eqref{rr2} with respect to $t$
and then putting $t=0$, we have
\begin{equation}\label{rr4}
\lambda w\,m'(w)=m(w)b(w)-B_0m(w).
\end{equation}
We claim that in fact equation \eqref{rr4} is {\it{equivalent}} to
\eqref{rr1}:
\begin{lemma}
    A mapping $m\in\Hol(\D,\A_*)$ is a solution of \eqref{rr1} if and only if it is a solution to \eqref{rr4}.
\end{lemma}

\begin{proof}
    We have already shown that \eqref{rr1} implies \eqref{rr4}. Now assume $m\in\Hol(\Omega,\A_*)$
    satisfies \eqref{rr4}. Define $\tilde{\gamma}_t$ by
    \begin{equation}
    \label{d1}\tilde{\gamma}_t(w)=m(e^{-\lambda t }w)^{-1}e^{tB_0}m(w).
    \end{equation}
    Multiplying both sides by $m(e^{-\lambda t }w)$, and differentiating with respect to $t$, we have
\begin{equation}\label{p1}
-\lambda e^{-\lambda t}w\,m'(e^{-\lambda t}w)\tilde{\gamma}_t(w)+ m(e^{-\lambda t
}w)\frac{d}{dt}\tilde{\gamma}_t(w)=e^{tB_0}B_0 m(w).
\end{equation}
Substituting $e^{-\lambda t}w$ in place of $w$ in \eqref{rr4} we get
   \begin{equation}\label{aux}\lambda e^{-\lambda t}w\,m'(e^{-\lambda t}w)=m(e^{-\lambda t}w)b(e^{-\lambda t}w)-B_0m(e^{-\lambda
   t}w).\end{equation}
   From \eqref{d1}--\eqref{aux}, we have
            $$-m(e^{-\lambda t}w)b(e^{-\lambda t}w)\tilde{\gamma}_t(w)+ m(e^{-\lambda t }w)\frac{d}{dt}\tilde{\gamma}_t(w)=0,$$
    that is,
            $$\frac{d}{dt}\tilde{\gamma}_t(w)=b(e^{-\lambda t}w)\tilde{\gamma}_t(w).$$
    We therefore see that the mapping $t\mapsto\tilde{\gamma}_t(w)$ satisfies the same differential equations as
    does $t\mapsto\gamma_t(w)$ (see \eqref{de1}), and since $\tilde{\gamma}_0(w)=\gamma_0(w)=1_{\A}$,
    the uniqueness of solutions for initial-value problems implies that $\tilde{\gamma}_t=\gamma_t$,
    hence by \eqref{d1} we have the representation \eqref{rr1}.
\end{proof}

In the following we thus investigate equation \eqref{rr4},
obtaining necessary conditions and sufficient conditions for its
solvability. Note that, in view of Lemma \ref{local}, it suffices
to solve (\ref{rr4}) in a neighborhood of $w=0$. Since \eqref{rr4}
is a first-order differential equation for $m(w)$ with a
singularity at $w=0$, we cannot appeal to the existence theorem
for differential equations to obtain local solvability (and indeed
the equation is not always solvable).

To proceed, we expand $b(w)$ and $m(w)$ in power series: $\displaystyle
b(w)=\sum_{k=0}^\infty b_k w^k$ and
\begin{equation}\label{MS}
m(w)=\sum_{k=0}^\infty m_k w^k,
\end{equation}
where all the coefficients are elements of $\A$. Equation
\eqref{rr4} then becomes
\[
\lambda\sum_{k=0}^\infty km_k w^k=\sum_{k=0}^\infty w^k\sum_{l=0}^k
m_lb_{k-l} - \sum_{k=0}^\infty B_0 m_k w^k.
\]
Equating corresponding coefficients, we get
$$k\lambda m_k=\sum_{l=0}^k m_lb_{k-l}- B_0 m_k$$
or, noting that $b_0=B_0$,
\begin{equation}\label{rec}
k\lambda  m_k-\left(m_kB_0-B_0 m_k\right)=\sum_{l=0}^{k-1}
m_lb_{k-l},\quad k\geq 0.
\end{equation}
We therefore have a sequence of recursive equations, with the
equation for $k=0$ holding trivially (since $m_0=1_{\A}$), and the
equation for $k\geq 1$, if it is uniquely solvable, defining
$m_k$. Thus the solvability of the sequence of equations
\eqref{rec} is a necessary condition for the existence of a
mapping $m$ satisfying \eqref{rr4}. Conversely, if the sequence of
equations \eqref{rec} is solvable, and if the series \eqref{MS}
can be shown to converge in a neighborhood of $w=0$, then it
defines a solution of \eqref{rr4} in a neighborhood of $0$, and,
by Lemma \ref{local}, on the whole of $\Omega$. We have therefore
shown:
\begin{propo}
The semicocycle generated by $B\in\Hol(\D,\A)$ is linearizable if
and only if the sequence of equations \eqref{rec} with
$$B_0+\sum_{k=1}^\infty b_k w^k=B(h^{-1}(w)),$$
is solvable.
\end{propo}

To study the solvability of \eqref{rec}, we define the bounded
linear operator $\ad_{B_0}:\A\rightarrow \A$ by
$$\ad_{B_0}(a)=aB_0-B_0a .$$
Then \eqref{rec} can be written as
\begin{equation}\label{rec1}
k\lambda m_k-\ad_{B_0}(m_k)=\sum_{l=0}^{k-1} m_lb_{k-l}.
\end{equation}
We therefore conclude that, for given $k$, a sufficient condition
for \eqref{rec1} to be solvable is that $k\lambda$ is not in the
spectrum of $\ad_{B_0}$. If we assume the condition
\begin{equation}
\label{nr} \sigma(\ad_{B_0})\cap \left(\lambda\N\right) = \emptyset,
\end{equation}
then all the equations \eqref{rec1} will be solvable:
\begin{equation}\label{rec2}
m_k=(k\lambda\cdot 1_\A -\ad_{B_0})^{-1}\sum_{l=0}^{k-1}
m_lb_{k-l},\quad k\in\N.
\end{equation}
It is known that the spectrum of $\ad_{B_0}$ can be related to the
spectrum of~$B_0$.
\begin{lemma}\label{ei} We have
    $$
    \sigma(\ad_{B_0})\subseteq \sigma(B_0)-\sigma(B_0):=
    \left\{\lambda_1-\lambda_2:\quad \lambda_1,\lambda_2\in \sigma(B_0)\right\}.
    $$
    If $\A$ is finite-dimensional, then the inclusion sign can be replaced by the equality sign.
\end{lemma}
For proof see \cite[Theorem 11.23]{R-73}, and for the
finite-dimensional case \cite[Exercise 5.15.1]{B-G}.

In view of this lemma, we can see that \eqref{nr} follows from the
condition
\begin{equation}\label{nr1}
\left(\sigma(B_0)-\sigma(B_0)\right)\cap \left(\lambda\N\right) =
\emptyset,
\end{equation}
and in the case in which $\A$ is finite-dimensional the two
conditions \eqref{nr} and \eqref{nr1} are equivalent.

We now show that if \eqref{nr} holds then not only are the
equations \eqref{rec1} solvable, but the series \eqref{MS}
corresponding to the solutions $\{m_k\}_{k=0}^\infty$ converges in
a neighborhood of $w=0$, and thus by Lemma~\ref{local} defines a
solution of \eqref{rr4}.

We note that the condition $\lambda k\not\in \sigma(\ad_{B_0})$
obviously holds for all $k>|\lambda|^{-1}\|\ad_{B_0}\|_{L(\A)}$. This
observation implies that, under (\ref{nr1}), there exists a
uniform bound
\[
\left\| (k\lambda\cdot 1_\A-\ad_{B_0} )^{-1}\right\| \leq C_1,\quad
k\in\N.
\]
Therefore
\begin{equation}\label{bound}
\|m_k\| \leq C_1 \sum_{l=0}^{k-1} \|m_l\|\cdot \| b_{k-l} \| ,
\quad k\in\N.
\end{equation}
Since $b$ is holomorphic around zero, we can choose $r,C_2>0$ so
that we have a bound
  $\displaystyle\|b_k\|\leq C_2 \cdot \frac{1}{r^k} $ for all $k\in \N$.
Setting $C_3=C_1C_2$, we have
  $$\|m_k\| \leq \frac{C_3}{r^k}\sum_{l=0}^{k-1} r^l \|m_l\|. $$
This can be used to prove, by induction, that
\begin{equation}\label{bm}
\|m_k\|\leq \left(\frac{C_3+1}{r}\right)^k.
\end{equation}
Indeed, since $m_0=1_{\A}$, \eqref{bm} holds for $k=0$, and,
assuming \eqref{bm} for $k=1,2,\ldots, n$, we have
  \begin{eqnarray*}
\|m_{n+1 }\| &\le& \frac{C_3}{r^{n+1}}\sum_{l=0}^n r^l \|m_l\| \le
\frac{C_3}{r^{n+1}}\sum_{l=0}^n (C_3+1)^l \\
&=& \frac{C_3}{r^{n+1}}\cdot \frac{(C_3+1)^{n+1}-1}{(C_3+1)-1} \le
\left(\frac{C_3+1}{r}\right)^{n+1} .
  \end{eqnarray*}
Thus the series \eqref{MS} converges at least when $|w|<\frac
r{C_3+1}$. In view of Lemma~\ref{local}, we have proven the main
result of this section.

\begin{theorem}\label{th1}
  Let a semigroup $\Ff$ be generated by $f\in\Hol(\D,\C)$ such that
  $f(z_0)=0$ for some $z_0\in\D$ and $\lambda:=-f'(z_0)$ satisfy $\Re\lambda>0$.
  Let $\{\Gamma_t\}_{t\geq 0}$ be a semicocycle over $\Ff$. If $B_0:= \left.\frac{d}{dt}\Gamma_t(z)\right|_{t=0}$ satisfies  condition
  \eqref{nr}, that is,
     $$\sigma(\ad_{B_0})\cap \left(\lambda\N\right) = \emptyset,$$
  then the semicocycle $\{\Gamma_t\}_{t\ge0}$ generated by $B$, is linearizable, that is, it can
  be represented in the form
  \[
\Gamma_t(z)=M(F_t(z))^{-1}e^{tB_0}M(z)
  \]
  with some $M\in\Hol(\D,\A_*)$.
\end{theorem}

Note that the equality $B_0=B(0)=0$ implies that $\ad_{B_0}=0$, so
that \eqref{nr} trivially holds, hence
$\Gamma_t(z)=M(F_t(z))^{-1}M(z)$. Thus we have the following.
\begin{corol}
A semicocycle $\{\Gamma_t\}_{t\ge0}$ is a coboundary if and only
if $B_0=\frac{d}{dt}\Gamma_t(z_0)\Big|_{t=0}=0$.
\end{corol}

\vspace{2mm}

We conclude with some remarks concerning Theorem~\ref{th1}:

(1) It is noteworthy that the sufficient condition given in the
theorem is only in terms of the value $B_0$ of $B$ at the interior
fixed point of the semigroup.

(2) We also recall that the condition \eqref{nr} follows from the
more `explicit' condition (\ref{nr1}) on the spectrum of $B_0$,
and that in the case that $\A$ is finite-dimensional, it is
equivalent to this condition.

(3) In the case that $\A$ is finite-dimensional, the condition
\eqref{nr1} holds `generically', in the sense that, given $B_0$,
the set of $\lambda$'s for which it does not hold is a countable set.

(4) In case $\A$ is commutative, we have $\ad_{B_0}=0$, so that
(\ref{nr}) trivially holds, and linearizability is assured. We
thus recapture the conclusion of Theorem \ref{th-present-Gamma}
above, but the value of that theorem is that, for commutative
algebras, it gives us a simple explicit representation for $M$,
which is not available in the non-commutative case.

(5) The condition \eqref{nr} is a sharp one, in the following
sense: assuming that $B_0\in \A$ is such that \eqref{nr1} does not
hold, we can always find function $B:\D\to\A$ such that
$B(0)=B_0$, and the semicocycle generated by $B$ is {\it not}
linearizable. Indeed, let $k$ be an integer such that $k\lambda\in
\sigma(\ad_{B_0})$. Then by Lemma \ref{ei} we have $k\lambda\in
\sigma(\ad_{B_0})$, hence (using the finite-dimensionality of
$\A$) there exists $A\in \A$ which is {\it{not}} in the range of
$k\lambda\cdot 1_\A-\ad_{B_0}$. Take $b(w)=B_0+Aw^k$ (that is,
$B(z)=B_0+h(z)^kA$). We thus have $b_j=0$ for $j\neq 0,k$. Then
equation \eqref{rec1} (for the chosen value of $k$) becomes
$$(k\lambda\cdot 1_\A -\ad_{B_0})(m_k)=A,$$
which is unsolvable, since $A$ is not in the range of $k\lambda\cdot
1_\A-\ad_{B_0}$. Therefore the sequence of equations \eqref{rec2}
is not solvable, so that the semicocycle generated by $B$ is not
linearizable.

(6) Using the above observation, we can easily construct simple
examples of nonlinearizable semicocycles. For instance, take the
semigroup $F_t(z)=e^{-t}z$, so that  $\lambda=-f'(0)=1$ and the
K{\oe}nigs function is the identity $h(z)=z$ with $\Omega=\D$.
Take $\A=M_2(\C)$. Let the semicocycle generator be
$$B(z)=\left(\begin{array}{cc}
\alpha_1 + \beta_{11}z& \beta_{12}z \\
\beta_{21}z & \alpha_2+\beta_{22}z
\end{array}  \right).$$
Then
$$b(z)=B(z)=b_0+ b_1 z, $$
where
$$b_0=B_0=\left(\begin{array}{cc}
\alpha_1 & 0\\
0 & \alpha_2
\end{array}  \right),\quad b_1=\left(\begin{array}{cc}
\beta_{11} & \beta_{12} \\
\beta_{21}  & \beta_{22}
\end{array}  \right).$$
By Theorem \ref{th1}, if $|\alpha_1-\alpha_2|$ is not a natural
number, then the semicocycle generated by $B$ is linearizable. Therefore, to obtain non-linearizable semicocycles, we can take, {\it{e.g.}},
$\alpha_2=\alpha_1+1$. Write
  $$m_1=\left( \begin{array}{cc}
  \eta_{11}  &  \eta_{12}\\
  \eta_{21} & \eta_{22}
  \end{array} \right).$$
Equation \eqref{nr1}, for $k=1$, is then
\begin{eqnarray*}
\left( \begin{array}{cc}
\eta_{11}  &  \eta_{12}\\
\eta_{21} & \eta_{22}
\end{array} \right)
-\left( \begin{array}{cc}
\eta_{11}  &  \eta_{12}\\
\eta_{21} & \eta_{22}
\end{array} \right) \left(\begin{array}{cc}
\alpha_1 & 0\\
0 & \alpha_2
\end{array}  \right)
+ \left(\begin{array}{cc}
\alpha_1 & 0\\
0 & \alpha_2
\end{array}  \right)\left( \begin{array}{cc}
\eta_{11}  &  \eta_{12}\\
\eta_{21} & \eta_{22}
\end{array} \right) \\
=\left(\begin{array}{cc}
\beta_{11} & \beta_{12} \\
\beta_{21}  & \beta_{22}
\end{array}  \right) \end{eqnarray*}
which, after simplifying and using the assumption
$\alpha_2=\alpha_1+1$, becomes
$$\left( \begin{array}{cc}
\eta_{11}  &  0\\
2\eta_{21} & \eta_{22}
\end{array} \right)=-\left(\begin{array}{cc}
\beta_{11} & \beta_{12} \\
\beta_{21}  & \beta_{22}
\end{array}  \right). $$
We thus see that the equation is not solvable whenever
$\beta_{12}\neq 0$. This explains Example~\ref{ex-not-rep} above.

\bigskip

{\bf Acknowledgments.} The authors are grateful to Lior Rosenzweig
for his fruitful remarks.  This work was partially supported by
the European Commission under the project STREVCOMS
PIRSES-2013-612669.



\end{document}